\documentclass[12pt]{article}
\usepackage{amsmath,amsthm,amssymb,graphicx}
\usepackage{algorithm,hyperref,lmodern}
\usepackage[a4paper,left=3.5cm,right=3.3cm,top=3.2cm,bottom=3cm]{geometry}
\usepackage{float}

\theoremstyle{plain}
\newtheorem{theorem}{Theorem}
\newtheorem{lemma}{Lemma}

\newtheorem{assume}{Assumption}
\newtheorem{remark}{Remark}
\newtheorem{definition}{Definition}
\allowdisplaybreaks

\begin{document}

\title{Pseudo-Bayesian Quantum
Tomography with Rank-adaptation}
\author{The Tien Mai
\footnote{Corresponding author. Email:  thetien.mai@ensae.fr   \&  pierre.alquier@ensae.fr}
\quad \&  Pierre Alquier
\\
{\small  CREST, ENSAE, Universit\'e Paris Saclay}
\\
{\small 3 av. Pierre Larousse,
92245 Malakoff CEDEX, France}
}

\maketitle

\begin{abstract}
Quantum state tomography, an important task in quantum
information processing, aims at reconstructing a state
from prepared measurement data. Bayesian methods are
recognized to be one of the good and reliable choices in
estimating quantum states~\cite{blume2010optimal}.
Several numerical works showed that Bayesian
estimations are comparable to, and even better than
other methods in the problem of $1$-qubit state
recovery. However, the problem of choosing prior
distribution in the general case of $n$ qubits is not
straightforward. More importantly, the statistical
performance of Bayesian type estimators have not
been studied from a theoretical perspective yet. In this
paper, we propose a novel prior for quantum states
(density matrices), and we define pseudo-Bayesian
estimators of the density matrix. Then, using
PAC-Bayesian theorems~\cite{catonibook}, we
derive rates of convergence for the posterior mean.
The numerical performance of these estimators are
tested on simulated and real datasets.
\end{abstract}

\section{Introduction}

Playing a vital role in quantum information
processing, as well as being fundamental for
characterizing quantum objects, quantum state
tomography focuses on reconstructing the (unknown)
state of a physical quantum system~\cite{MR2181445bookParis},
 usually represented by the so-called density matrix
$\rho$ (the exact definition of a density matrix
is given in Section~\ref{set-up}).
This task is done by using outcomes of measurements performed on
many independent systems identically prepared in the
same state.

The 'tomographic' method, also named as linear/direct
inversion~\cite{vogel1989determination,vrehavcek2010operational},
is the simplest and oldest estimation procedure. It is actually
the analogous of the least-square estimator in the quantum setting.
Although
easy in computation and providing unbiased
estimate~\cite{PhysRevLett.114.080403},
it does not generate a physical density matrix as an
output~\cite{shang2014quantum}. Maximum likelihood
estimation~\cite{hradil20043} is the current procedure
of choice. Unfortunately, it has some critical
flaws detailed in~\cite{blume2010optimal}, including
a huge computational complexity. Furthermore,
both these methods are not adaptive to the case where
a system is in a state $\rho$ for which some
additional information is available. Note especially that,
physicists focus on so-called pure states,
for which ${\rm rank}(\rho)=1$.

The problem of rank-adaptivity was tackled thanks to
adequate penalization. Rank-penalized maximum
likelihood (BIC) was introduced
in~\cite{guctua2012rank} while a rank-penalized
least-square estimator $\hat{\rho}_{{\rm rank-pen}}$
was proposed in~\cite{alquier2013rank}, together with
a proof of its consistency. More specifically, when
the density matrix of the system is $\rho^0$ with
$ r = {\rm rank}(\rho^0) $, the authors
of \cite{alquier2013rank} proved that the
Frobenius norm of the estimation error satisfies
$\|\hat{\rho}_{{\rm rank-pen}} - \rho^0 \|_F^2
=\mathcal{O}(r 4^n /N)$ where $N$ is the number
of quantum measurements. The rate was improved to
$\mathcal{O}(r3^n/N)$ by~\cite{butucea2015spectral},
using a thresholding method. Note that the rate
$\mathcal{O}(r2^n/N)$ was first claimed in the paper,
but in the Corrigendum~\cite{butucea2016corrigendum},
the authors acknowledge that this is not the case.
The paper however contains a proof that
no method can reach a rate smaller than $r2^n / N$.
So, the minimax-optimal rate is somewhere in between
$r 2^n / N$ and $r3^n / N$.

Note that all the aforementioned papers
only cover the complete measurement case
(the definition is given in Section~\ref{set-up},
basically it means that we have observations for all the
observables given by the Pauli basis).
The statistical relationship between matrix completion
and quantum tomography with incomplete measurements (in
the Le Cam paradigm) has been investigated in~\cite{wang2013}.
Thus compressed sensing ideas have been successfully
proposed in estimating a density state from incomplete
measurements~\cite{gross2010quantum,gross2011recovering,
flammia2012quantum,koltchinskii2011}.

On the other hand, Bayesian estimation has been
considered in this context. The papers~\cite{buvzek1998reconst,
Baier_comparison} compare Bayesian methods to other methods
on simulated data. More recently,~\cite{kravtsov2013experimental,
ferrie2014quantum,ferrie2015have,schmied2014quantum}
discuss efficient algorithms for computing Bayesian estimators.
Importantly, \cite{blume2010optimal}
showed that Bayesian method comes with natural error bars and is
the most accurate scheme w.r.t. the expected error
(operational divergence) (even) with finite samples.
However, there is no theoretical guarantee on the
convergence of these estimators.

More works on quantum state tomography in various settings
include~\cite{audenaert2009quantum,carlen2010trace,RauGibbsstate,rau2014Gibbstate,
ferrie2014likelihood_free}.

In this paper, we consider a pseudo-Bayesian estimation,
where the likelihood is replaced by pseudo-likelihoods
based on various moments (two estimators,
corresponding to two different pseudo-likelihood, are
actually proposed). Using PAC-Bayesian
theory~\cite{STW,McA,catoni2004statistical,catonibook,
dalalyan2008aggregation,suzuki2012pac},
we derive oracle inequalities for the pseudo-posterior
mean. We obtain rates of convergence for these
estimators in the complete measurement setting.
One of them has a rate as good as the best
known rate up to date $\mathcal{O}( {\rm rank}(\rho^0) 3^n / N)$
(still, the other one is interesting for computationnal reasons
that are discussed in the paper).

The rest of the paper is organized as follow.  We recall
the standard notations and basics about quantum theory
in Section \ref{set-up}. Then the definition of the prior
and of the estimators are presented in Section~\ref{Def Prior}.
The statistical analysis of the estimators are given in
Section~\ref{MAIN}, while all the proofs are delayed to
the Appendix~\ref{Appendix}. Some numerical
experiments on simulated and real datasets are given in
Section~\ref{num}.

\section{Preliminaries}
\label{set-up}

\subsection{Notations}

A very good introduction to the notations and problems of
quantum statistics is given in~\cite{artiles}. Here, we
only provide the basic definitions required for the paper.

In quantum physics, all the information on the physical
state of a system can be encoded in its {\it density
matrix} $\rho$. Depending on the system in hand, this
matrix can have a finite or infinite number of entries.
A two-level system of $n$-qubits is represented by
a $2^{n}\times 2^{n}$ density matrix $\rho$, with
coefficients in $\mathbb{C}$. For the sake of simplicity,
the notation $d=2^n$ is used in~\cite{butucea2015spectral},
so note that $\rho$ is a $d\times d$ matrix.
This matrix is Hermitian
$\rho^\dagger=\rho$ (i.e. self-adjoint), semidefinite
positive $\rho\geq 0$ and has ${\rm Trace}(\rho)=1$.
Additionally, it often
makes sense to assume that the rank of $\rho$ is
small \cite{gross2010quantum,gross2011recovering}.
In theory, the rank can be any integer between $1$ and
$2^n$, but physicists are especially interested in pure
states and a pure state $ \rho $ can be defined by
${\rm rank}(\rho)=1$.

The objective of quantum tomography
is to estimate $\rho$ on the basis  of
experimental observations of many independent
and identically systems prepared in the state $\rho$ by
the same experimental device.

For each particle (qubit), one can measure one of
the three Pauli observables $\sigma_x, \, \sigma_y, \, \sigma_z$.
The outcome for each will be $ 1 $, or $ -1 $, randomly
(the corresponding probability depends on the state
$\rho$ and will be given in~\eqref{prob.fomala} below).
Thus for a $ n $-qubits system, we consider $3^n$
possible experimental observables. The set of all
possible performed observables is
\begin{align*}
\{\sigma_{\mathbf{a}} = \sigma_{{a}_1} \otimes
\ldots \otimes \sigma_{{a}_n}; \,  \mathbf{a} =
(a_1,\ldots,a_n) \in \mathcal{E}^n := \{x,y,z\}^{n}\},
\end{align*}
where vector $\mathbf{a} $ identifies the experiment.
The outcome for each fixed observable
setting will be a random vector
$ \mathbf{s} = (s_1, \ldots , s_n)   \in \mathcal{R}^n :=\{-1,1\}^{n}$,
thus there are  $ 2^n $ outcomes in total.

Let us denote $R^{\mathbf{a}}$ a $\mathcal{R}^n$-valued
random vector that is the outcome of an experiment
indexed by $\mathbf{a}$.
From the basic principles of quantum mechanics
(Born's rule), its probability distribution is given by
\begin{equation}
\label{prob.fomala}
\forall \mathbf{s} \in \mathcal{R}^n,
p_{\mathbf{a},\mathbf{s}} := \mathbb{P}
 (R^\mathbf{a}= \mathbf{s})
  = {\rm Trace} \left(\rho \cdot
 P_{\mathbf{s}}^{\mathbf{a}} \right),
\end{equation}
where  $ P_{\mathbf{s}}^{\mathbf{a}} :=
P_{s_1}^{a_{1}}\otimes \dots \otimes P_{s_n}^{a_n}$ and
$P_{s_i}^{a_i}$ is the orthogonal
projection associated to the eigenvalue $s_i$ in
the diagonalization of $ \sigma_{a_i} $
for $a_i\in \{x,y,z\}$ and $ s_i\in \{-1,1\} $ --
that is  $ \sigma_{a_i} = -1P^{a_i}_{-1}+1P^{a_i}_{+1}$.

The quantum state tomography problem is as follows: a physicist
has access to an experimental device that produces $n$-qubits in
a state $\rho^0$, and $\rho^0$ is assumed to be unknown. He/she
can produce a large number of replications of the $n$-qubits and
wants to infer $\rho^0$ from this.

In the complete measurement case, for {\it each}
experiment setting $ \mathbf{a}\in\mathcal{E}^n$, the experimenter
repeats $ m $ times the experiment corresponding to $\mathbf{a}$
and thus collects $m$ independent random copies of $R^\mathbf{a}$,
say $R^\mathbf{a}_1,\dots,R^\mathbf{a}_m$. As there are $3^n$
possible experiment settings $ \mathbf{a}$, we define the
\textbf{quantum sample} size as $ N:=m\cdot 3^n $. We will refer to
$(R^\mathbf{a}_i)_{i\in\{1,\dots,m\},\mathbf{a}\in\mathcal{E}^n}$
as $\mathcal{D}$ (for data).

Note that the case where we would only have access to experiments
$ \mathbf{a}\in\mathcal{A}$ where $\mathcal{A}$ is some proper
subset of $\mathcal{E}^n$ ($A \varsubsetneq \mathcal{E}^n$) is referred
to as the incomplete measurement case. In this paper, we focus on
the complete measurement case, but the extension to the incomplete
case is discussed in Section~\ref{conclusion}.

\subsection{Popular estimation methods}

A natural idea is to define the empirical frequencies
$$ \hat{p}_{\mathbf{a},\mathbf{s}}
= \frac{1}{m}\sum_{i=1}^m \mathbf{1}_{\{R_i^\mathbf{a}=\mathbf{s}\}}.
$$
Note that $ \hat{p}_{\mathbf{a},\mathbf{s}}$ is an unbiased estimator
of the probability $ p_{\mathbf{a},\mathbf{s}} $.
The inversion method is based on solving the linear
system of equations
\begin{equation} \label{rhohat}
\left\{ \begin{array}{l}
\hat p_{\mathbf{a},\mathbf{s}} = {\rm Trace}
\left(\hat {\rho} \cdot
 P_{\mathbf{s}}^{\mathbf{a}} \right),
 \\
 \mathbf{a}\in\mathbb{E}^n,
 \\
 \mathbf{s} \in \mathcal{R}^n.
 \end{array}
 \right.
\end{equation}
As mentioned above, the computation of $\hat{\rho}$ is quite
straighforward. Explicit formulas are classical, see
e.g.~\cite{alquier2013rank}.

Another commonly used method is maximum likelihood (ML) estimation,
where the likelihood is
\begin{equation*}
\mathcal{L}(\rho;\mathcal{D})  \propto \prod_{\mathbf{a}\in\mathbb{E}^n}
 \prod_{\mathbf{s} \in \mathcal{R}^n}  [{\rm Trace} \left(\rho \cdot
 P_{\mathbf{s}}^{\mathbf{a}} \right)]^{ n_{\mathbf{a},\mathbf{s}} },
\end{equation*}
where $n_{\mathbf{a},\mathbf{s}} = m \hat{p}_{\mathbf{a},\mathbf{s}}$
is the number of times we observed output $\mathbf{s}$ in experiment
$\mathbf{a}$ (obviously, $\sum_{\mathbf{s}} n_{\mathbf{a},\mathbf{s}} = m$).
As mentioned in the introduction,
both methods suffer many drawbacks. The inversion method returns
a matrix $\hat{\rho}$ that usually does not satisfy the axioms of
a density matrix. ML becomes expensive (inpractical) for
$ n \geq 10 $. Moreover, these
two methods can not take advantage of a prior knowledge
(e.x. low-rank state).

Considering the expansion of the density matrix $ \rho $
in the $n$--Pauli basis, i.e. $ \mathcal{B}= \{ \sigma_b =
\sigma_{{b}_1} \otimes \ldots \otimes \sigma_{{b}_n}, b
\in \{ I,x,y,z \}^n \}, \sigma_I = I ,$
\begin{align}
\label{Pauli expansion}
\rho = \sum_{b\in \{ I,x,y,z \}^n }\rho_b \sigma_b.
\end{align}
One can also estimate the density matrix via
estimating the coefficients in the Pauli expansion.
This was studied in~\cite{cai2015Pauli} where the authors
also make a sparsity assumption: that is, most of $\rho_b $
are small or very close to $0$. Note that, this is not related
to the setting we explore (low-rank assumption).

We now turn to the definition of a prior distribution
on density matrices that will allow to perform
(pseudo-)Bayesian estimation.

\section{Pseudo-Bayesian estimation and prior distribution on density matrices}
\label{Def Prior}

\subsection{Peudo-Bayesian estimation}

We remind that the idea of Bayesian statistics is to encode the
prior information on density matrices through a prior distribution
$\pi({\rm d}\rho)$. Inference is then done through the posterior
distribution $\pi({\rm d}\rho|\mathcal{D})
\propto \mathcal{L}(\rho) \pi({\rm d}\rho)$. Here, for computational
reasons, we replace the likelihood by a pseudo-likelihood.
This is an increasingly popular method in Bayesian
statistics~\cite{bissiri2013general}
and in machine
learning~\cite{catonibook,alquier2015properties,begin2016pac}.
We define
the pseudo-posterior by
\begin{align}
\tilde{\pi}_{\lambda}({\rm d}\nu)
\propto \exp\left[-\lambda \ell(\nu,\mathcal{D}) \right]
 \pi({\rm d}\nu),
 \label{pseudo_post}
\end{align}
the pseudo-likelihood being $\exp\left[-\lambda \ell(\nu,\mathcal{D}) \right]$.
The term $\ell(\nu,\mathcal{D}) $ can be specified by the user. Two examples
are provided in Section~\ref{MAIN}.
As a replacement of the likelihood, this term plays the role of the empirical
evidence. More specially
\begin{itemize}
 \item the role of $\exp\left[-\lambda \ell(\nu,\mathcal{D}) \right]$
 is to give more weight to the density $\nu$  when it
fits the data well;
 \item the role of $\pi({\rm d}\nu)$, the prior, is to
 restrict the posterior to the space of densities
 (and even give more weight to low-rank matrices if needed);
 \item $\lambda>0$ is a free parameter that allows to tune the balance
 between evidence from the data and prior information.
\end{itemize}
We finally define the pseudo-posterior mean (also refered to as
Gibbs estimator, PAC-Bayesian estimator or EWA, for exponentially weighted
aggregate~\cite{catonibook,dalalyan2008aggregation}):
\begin{align*}
\tilde{\rho}_{\lambda}
= \int \nu \tilde{\pi}_{\lambda}(d\nu).
\end{align*}
The definition of the estimator $\tilde{\rho}_{\lambda} $
based on the pseudo-posterior
$\tilde{\pi}_{\lambda}$ is actually validated by the
theoretical results from Section~\ref{MAIN}.

\subsection{Definition of the prior}

In the single qubit state estimation $ n=1$, the
representation of the quantum constraints is
explicit \cite{Baier_comparison,schmied2014quantum}.
Thus, one can place a prior distribution on the polar
reparametrization of the density. Up to our knowledge,
this has not been extended to the case $n>1$, and this
extension seems not straightforward.
For general n-qubit densities,
uninformative priors (e.g the Haar measure) are put
on $ \psi_{d\times K} $ matrices ($K\geq d$) and the
density state is built by $ \rho = \psi_{d\times K}
\psi^{\dagger}_{d\times K} $
 \cite{struchalin2016experimental,granade2015practical,
 huszar2012adaptive,ferrie2015have,zyczkowski2011generating}.
 One could also define a prior on the coefficients $\{\rho_b\}$
 of $\rho$ on the Pauli basis.
Nevertheless, none of these approaches seem helpful for
rank adaptation.

The idea for our prior is inspired by the priors used for low-rank
matrix estimation in machine learning, e.g.~\cite{mai2015,cottet20161}
and the references therein. Hereafter, we describe in details
the prior construction.

Let $V$ be a  vector in  $ \mathbb{C}^{d\times 1}\setminus \{\mathbf{0}\}$
($ d=2^n $ in our model), then $VV^{\dagger}$ is
a Hermitian, semi-definite positive matrix in
$ \mathbb{C}^{d\times d}$ with ${\rm rank}(VV^{\dagger}) = 1$.
Additionally, we can normalize $V$ (that is replace
$V$ by $V/\|V\|$), this lead to ${\rm Trace}
(VV^{\dagger}) = 1$. So, $VV^{\dagger}$
 satisfies the conditions of a density matrix
 (with rank-$1$).

Now, let $V_1, \ldots, V_d$ be $d$ normalized vectors
in $\mathbb{C}^{d\times 1}\setminus \{\mathbf{0}\}$ and $\gamma_1,\ldots
,\gamma_d$ be non-negative weights with
$ \sum_{j=1}^d\gamma_j =1$. Put
\begin{equation}
\label{formula}
\nu = \sum_{i=1}^d \gamma_i V_i V_i^{\dagger} .
\end{equation}
Then $ \nu $ is clearly a density matrix:
it is Hermitian (as a sum of Hermitian matrices), it is
semi-definite positive (same reason) and
$$ {\rm Tr}(\nu) = \sum_{i=1}^d \gamma_i {\rm Tr}
(V_i V_i^{\dagger}) =1 .$$

Moreover, note that any density matrix can be written
in such way, as we know that for any density matrix $\rho$,
\begin{equation}
\label{formula-diag}
\rho= U \Lambda U^{\dagger}
\end{equation}
and just write $U=(U_1|\dots|U_d)$ with
 the $U_i$'s being {\it orthogonal},
 where  $ \Lambda = {\rm diag}
(\Lambda_1, \ldots, \Lambda_n): \Lambda_1
\geq \ldots \geq \Lambda_n \geq 0,
\sum_{i = 1}^d  \Lambda_i = 1$.

The only difference in~(\ref{formula}) is that we do not
require that the $V_i$'s are orthogonal. Thus, it is
easier to simulate a matrix $\rho$ by simulating the $V_i$'s
and $\gamma_i's$ in~\eqref{formula} than by simulating
$U$ and $\Lambda$ in~\eqref{formula-diag}. Also, note that the
$\gamma_i$'s are not necessarily the eigenvalues of $\rho$.

\begin{definition}
We define the prior definition on $\rho$, $\pi({\rm d}\rho)$,
by
\begin{align*}
 V_1,\ldots, V_d & \sim \text{ i.i.d}  \text{ uniform distribution on
 the unit sphere,} \\
 (\gamma_1, \ldots,\gamma_d) & \sim \mathcal{D}ir(\alpha_1,\dots,\alpha_d),
 \\
 \rho & = \sum_{i=1}^d \gamma_i V_i V_i^{\dagger}
\end{align*}
where $\mathcal{D}ir(\alpha_1,\dots,\alpha_d)$ is the Dirichlet distribution
with parameters $\alpha_1,\dots,\alpha_d>0$.
\end{definition}

\begin{remark}
To get an approximate rank-$1$ matrix $\rho$, one can
take  all parameters of the Dirichlet distribution equal
to a constant that is very closed to 0 (e.g
$ \alpha_1=\ldots=\alpha_d=\frac{1}{d} $). And a typical
drawing will lead to one of the $\gamma_i's $ close to $1 $
and the others close to $0 $. See~\cite{wallach2009rethinking} for more
discussion on choosing the parameters for Dirichlet
distribution. Theoretical recommendations for the
$ \alpha_i $'s are given in Section~\ref{MAIN} below.
\end{remark}

\begin{remark}
We could impose the $V_i$'s to be orthogonal in practice. The
theoretical results would be unchanged, however, the implementation
of our method would become trickier. Note that to sample from the
uniform distribution on the sphere is rather easy. We can for example
simulate $\tilde{V}_i$ from any isotropic distribution, e.g.
$\mathcal{N}(0,\mathbb{I})$ and define $V_i := \tilde{V}_i / \|\tilde{V}_i\|$.
\end{remark}

\section{PAC-Bayesian estimation and analysis}
\label{MAIN}

\subsection{Pseudo-likelihoods}
Here, we consider two natural ways to compare a theoretical
density $\rho $ and the observations: first $p_{\mathbf{a},\mathbf{s}} $ should be close to the empirical part
$\hat{p}_{\mathbf{a},\mathbf{s}} $; second $\rho $ should be
close to the least square (invert) estimator $\hat{\rho} $.
As we have no reason to prefer one in advance, we define and
study $2 $ estimators.
\subsubsection*{a) Distance between the probabilities: prob-estimator}
We consider
$$ \ell^{prob}(\nu,\mathcal{D}) = \sum_{\mathbf{a}\in\mathcal{E}^n}
\sum_{\mathbf{s}\in \mathcal{R}^n}
\left[{\rm Tr}(\nu P_\mathbf{s}^\mathbf{a})
-\hat{p}_{\mathbf{a},\mathbf{s}}\right]^2$$
and
\begin{align*}
\tilde{\rho}^{prob}_{\lambda}
&= \int \nu \tilde{\pi}^{prob}_{\lambda}(d\nu),
\\
\tilde{\pi}^{prob}_{\lambda}({\rm d}\nu)
& \propto \exp\left[-\lambda \ell^{prob}(\nu,\mathcal{D})
 \right] \pi({\rm d}\nu).
\end{align*}
Note that if we use the shortened notation
$p_\nu=[{\rm Tr}(\nu P_\mathbf{s}^
\mathbf{a})]_{\mathbf{a},\mathbf{s}}$ and
$\hat{p}=[\hat{p}_{\mathbf{a},
\mathbf{s}}]_{\mathbf{a},\mathbf{s}}$
 then
$$ \ell^{prob}(\nu,\mathcal{D}) = \| p_\nu -
\hat{p}  \|^2_F $$
(Frobenius norm).
This distance quantifies how far the probabilities and the empirical
frequencies in the sample are.

\subsubsection*{b) Distance between the density matrices: dens-estimator}
Now, let us take:
$$ \ell^{dens}(\nu,\mathcal{D}) =
 \|\nu-\hat{\rho}\|_F^2. $$
and
\begin{align*}
\tilde{\rho}^{dens}_{\lambda}
&= \int \nu \tilde{\pi}^{dens}_{\lambda}(d\nu),
\\
\tilde{\pi}^{dens}_{\lambda}({\rm d}\nu)
& \propto \exp\left[-\lambda \ell^{dens}(\nu,
\mathcal{D}) \right] \pi({\rm d}\nu).
\end{align*}
In another words, this estimator finds a balance between
prior information and closeness to the least square
estimate $\hat{\rho}$. From a computational point of
view, this estimator is easier to implement than the
previous estimator.

\subsection{Statistical properties of the estimators}

\begin{assume}
\label{A1}
Fix some constants $D_1>0$ and $D_2>0$ (that do not depend on
$m$ nor $n$).
We assume that the parameters of the Dirichlet
prior distribution
$ \mathcal{D}ir(\alpha_1,\dots,\alpha_d) $ satisfy
\begin{itemize}
 \item $\forall i = 1,\ldots,d : \alpha_i \leq 1$,
 \item $  \sum_{i=1}^d\alpha_i =D_1$,
 \item  $ \prod_{i=1}^d \alpha_i \geq e^{-D_2d \log(d)} $.
\end{itemize}
\end{assume}

Note that this assumption is satisfied for
$ \alpha_1 = \ldots = \alpha_d = 1/d$ with
$ D_1 = D_2 = 1$.

The first theorem provides the concentration bound on
the square error of the first estimator
 $\tilde{\rho}^{prob}_{\lambda}$.
The proof of this theorem is left to the
\hyperref[Appendix]{appendix}.

\begin{theorem}
\label{thrm: rate 1}
Fix a small $\epsilon \in (0,1) $.
Under \hyperref[A1]{Assumption 1}, for
$ \lambda = \lambda^*: =m/2 $, with probability
at least $ 1-\epsilon$, one has
\begin{align*}
 \| \tilde{\rho}^{prob}_{\lambda^*} - \rho^0\|_F^2
 \leq
  C^{prob}_{D_1,D_2}\frac{   3^n {\rm rank}(\rho^0)
  \log\left(
\frac{{\rm rank}(\rho^0) N}{2^n}
\right)
 + (1.5)^n\log(2/\epsilon)}{ N},
\end{align*}
where $ C^{prob}_{D_1,D_2}$ is a constant that depends only
on $D_1,D_2$.
\end{theorem}

\begin{remark}
As said in the introduction, the best known rate up-to-date in
this problem is $\frac{   3^n {\rm rank}(\rho^0)}{N}$, so our estimator
$ \tilde{\rho}^{prob}_{\lambda^*}$ reaches this rate (up to log terms).
This rate is actually $\left(\frac{3}{2}\right)^n
\frac{   rd}{N}$ and the best lower bound known in this case is
$\frac{rd}{N}$~\cite{butucea2015spectral} (we remind that $d=2^n$).
\end{remark}

The next theorem presents the square error bound of
the second estimator $\tilde{\rho}^{dens}_{\lambda}$.
Here again, see the appendix for the proof.

\begin{theorem}
\label{rate 2}
Fix a small $\epsilon \in (0,1) $.
Under \hyperref[A1]{Assumption 1}, for
$ \lambda = \lambda^*: = \frac{N}{5^{n}4}$, with probability
at least $ 1-\epsilon$,
\begin{equation}
\|\tilde{\rho}^{dens}_{\lambda^*} -\rho^0\|_F^2
\leq
C^{dens}_{D_1,D_2} \frac{10^n {\rm rank}(\rho^0) \log\left(
\frac{{\rm rank}(\rho^0) N}{2^n}
\right)
+ 5^n \log(2/\varepsilon)}{N}
\end{equation}
where $ C^{dens}_{D_1,D_2}$ is a constant that depends only
on $D_1,D_2$.
\end{theorem}

\begin{remark}
The guarantee for $\tilde{\rho}^{dens}_{\lambda^*}$ is far less satisfactory.
However, as this estimator is easier to compute, we think it is
interesting to provide a convergence rate, even if it is far from optimal:
note that for a fixed $d$, the bound goes to $0$ when $m\rightarrow \infty$.
\end{remark}

\begin{remark}
Experiments show that
$ \lambda = \lambda^*: = \frac{N}{5^{n}4}$
is actually not the best choice for dens-estimator.
The choice $ \lambda = \frac{N}{4}$ (heuristically motivated
by~\cite{dalalyan2008aggregation}) leads to results
comparable to the prob-estimator in Section
\ref{num}. This leads to the conjecture that the rate of
$ \tilde{\rho}^{dens}_{N/4} $ is much better than
$  \frac{10^n {\rm rank}(\rho^0) }{N} $
but this is still an open question.
\end{remark}

\section{Numerical Experiments}
\label{num}

\subsection{Metropolis-Hastings Implementation}
We implement the two proposed estimators via the
Metropolis-Hasting (MH) algorithm~\cite{robert2013monte}.
Note that to draw $  (\gamma_1, \ldots,\gamma_d)
\sim \mathcal{D}ir(\alpha,\dots,\alpha) $ is equivalent to
draw $ \gamma_i = Y_i/(Y_1+\ldots+Y_d) $ with
$ Y_i \overset{i.i.d}{\sim} Gamma(\alpha,1), \forall i = 1,\ldots,d $.
Thus, instead of $\gamma_i's $, we conduct a MH updating for $ Y_i's $.
So the objective is to produce a Markov chain
$(Y^{(t)}_1,\ldots,Y^{(t)}_d, V^{(t)}_1,\ldots, V^{(t)}_d) $.
From this, we deduce obviously the sequence $(\gamma^{(t)}_1,\ldots,\gamma^{(t)}_d, V^{(t)}_1,\ldots, V^{(t)}_d) $ and use the following empirical mean as the Monte-Carlo approximation of our estimator:
$$
\hat{\rho}^{{\rm MH}} := \frac{1}{T}\sum_{t=1}^T\left(  \sum_{i=1}^d \gamma_i^{(t)} V_{i}^{(t)} (V_{i}^{(t)})^\dagger\right).
$$

\begin{algorithm}
\caption{MH implementation}
For $ t$ from $ 1$ to $ T$, we iteratively update through the following steps:
\begin{description}
\item[updating for $Y_i's $:] for $i $ from $1 $ to $d $,
\\
Sample $\tilde{Y}_i  \sim h(y|Y^{(t-1)}_i) $
where $h$ is a proposal distribution given explicitely below.
\\
Calculate $\tilde{\gamma_i} = \tilde{Y}_i / (\sum_{i=1}^d\tilde{Y}_i) $.
\\
Set
$$ Y^{(t)}_i  =
\begin{cases}
\tilde{Y}_i    &\text{with probability }\min \left\{ 1,R(\tilde{Y}, Y^{(t-1)})\right\},
\\
 Y^{(t-1)}_i        & \text{otherwise} 
\end{cases} $$
where $R(\tilde{Y}, Y^{(t-1)})$ is the acceptance ratio given below.
\\
Put $\gamma_i^{(t)} =Y^{(t)}_i /(\sum_{j=1}^dY^{(t)}_j) , i = 1,\ldots,d$.

\item[updating for $V_i's $:] for $i $ from $1 $ to $d $,
\\
Sample $ \tilde{V}_i $ from the uniform distribution on
 the unit sphere.
\\
Set
$$ V^{(t)}_i  =
\begin{cases}
\tilde{V}_i    &\text{with probability }   \min \{ 1,A(V^{(t-1)}, \tilde{V}) \} ,
\\
 V^{(t-1)}_i        & \text{otherwise}  ,
\end{cases} $$
where $A(V^{(t-1)}, \tilde{V})$ is the acceptance ratio given below.
\end{description}
\end{algorithm}

Let us now give precisely $h$, $R$ and $A$.
We define $ h(\cdot|Y^{(t-1)}_i)$ as the probability distribution of $U = Y^{(t-1)}_i \exp(y)$ where $y \sim \mathcal{U}(-0.5,0.5) $.
Following~\cite{robert2013monte} the acceptance ratios are then given by:
\begin{align*}
\log(R( \tilde{Y}, Y^{(t-1)}))
&= \lambda \ell\left(\sum_{i=1}^d \tilde{\gamma_i}
 V_i V_i^{\dagger} ,\mathcal{D}\right)
- \lambda \ell\left(\sum_{i=1}^d \gamma^{(t-1)}_i
V_i V_i^{\dagger} ,\mathcal{D}\right)
\\
& + \sum_{i=1}^d((\alpha-1)\log(\tilde{Y}_i) -\tilde{Y}_i )
-
\sum_{i=1}^d((\alpha-1)\log(Y_i^{(t-1)}) -Y_i^{(t-1)})
\\
& + \sum_{i=1}^d
\tilde{Y}_i - \sum_{i=1}^d Y_i^{(t-1)}
\end{align*}
and
\begin{align*}
\log (A(V^{(t-1)}, \tilde{V})) 
= 
\lambda \ell\left(\sum_{i=1}^d \gamma_i
 \tilde{V}_i \tilde{V}_i ^{\dagger} ,\mathcal{D}\right)
- \lambda \ell\left(\sum_{i=1}^d \gamma_i
V^{(t-1)}_i (V^{(t-1)}_i)^{\dagger} ,\mathcal{D}\right)
\end{align*}
where $\ell(\cdot,\mathcal{D})$ stands for $\ell^{dens}(\cdot,\mathcal{D})$
or $\ell^{prob}(\nu,\mathcal{D})$ depending on the estimator we are
computing.

\subsection{Experiments}
We study the numerical performance
of the prob-estimators with $ \lambda = m/2 $, i.e.
$\tilde{\rho}_{m/2}^{prob}$ and the
dens-estimator with $ \lambda =  \frac{N}{4} $, i.e.
$\tilde{\rho}_{ N/4}^{dens}$ on the
following settings, all with $ n=2,3,4$ ($d=4,8,16 $):
\begin{itemize}
\item a pure state density (rank-$ 1$)
$\rho = \psi \psi^{\dagger}$ with $ \psi \in \mathbb{C}^{d\times 1} $,
\item a rank-$2 $ density matrix that
$ \rho_{rank-2} = \frac{1}{2}\psi_1 \psi_1^{\dagger}
+  \frac{1}{2}\psi_2 \psi_2^{\dagger}$ with
 $ \psi_1,\psi_2  $ being two normalized orthogonal
 vectors in $  \mathbb{C}^{d\times 1} $,

\item an ``approximate rank-$2$" density matrix:
$ \rho = w \rho_{rank-2} + (1-w)\frac{\mathbb{I}_d}{d}, w=0.98 $. Note that by ``approximate rank-$2$", we mean that
$\rho $ is very well approximated by a rank-$2$ matrix
$\rho_{rank-2} $ (in the sense that
$\|\rho- \rho_{rank-2} \|_F^2 $ is small), but in general
$\rho $ itself is full rank,

\item a maximal mixed state (rank-$d$).
\end{itemize}

The experiments are done for $ m=20;200;1000;2000 $.
The parameter for $\mathcal{D}ir(\alpha,\dots,\alpha) $
is $ \alpha=0.5 $. We repeat each experiment $10$ times,
and compute the mean of the square error, MSE,
 $\|\hat{\rho}-\rho\|_F^2$ for each estimator, together with 
 the associated standard deviation (between brackets
in Tables~\ref{tablen_4},\ref{tablen_3},\ref{tablen_2}).

\subsection{Results}

We compare the prob- and dens-estimator to the simple
inversion procedure and to the thresholding estimator
of \cite{butucea2015spectral}. The results are given in
Tables~\ref{tablen_4},\ref{tablen_3},\ref{tablen_2}
(outputs from the
\textbf{R} software). The conclusions are:
\begin{itemize}
\item The prob-estimator seems to be the most accurate
but also comes with a larger standard deviation. This
might be due to slow convergence of the MCMC
procedure. Indeed each step is computationally highly
expensive.
\item The dens-estimator is easier to compute and while it is
less accurate than the prob-estimator, it
still shows better results than the direct inversion
method.
\item The thresholding estimator of
\cite{butucea2015spectral} works well for rank-1
states but seems to bring too much bias for other states.
\end{itemize}

Besides the square error, the eigenvalues of the
estimates are also important when reconstructing
density matrices. In Figure~\ref{eigenvalues_simulated},
the dens-estimator returns with eigenvalues similar to
the true eigenvalues of the true density matrix, while
the prob-estimator seems not to shrink enough.

\begin{figure}[!ht]
\centering
\includegraphics[scale=.47]{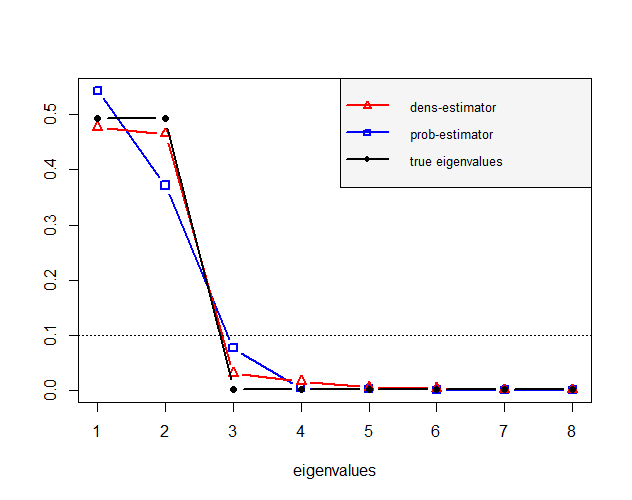}
\caption{Eigenvalues of estimates for an ``approximate
  rank-2" density with $ d=2^3, m = 200 $.}
\label{eigenvalues_simulated}
\end{figure}

\begin{table}[H]
\caption{MSEs for $ n=4 $
  (together with standard deviations)}
\label{tablen_4}
\begin{tabular}{| l | c | c | c | c |}
\hline\hline
                 & $m = 20$ & $m = 200$ & $m = 1000$ & $m = 2000$
\\
\hline
\multicolumn{5}{|c|}{pure state, MSEs$\times 10^{5}$}
\\
\hline
Inversion & 175 (4e-4) & 14.8 (2e-5) & 2.71 (8e-6) & 1.55 (5e-6)
\\
\hline
Thresholding & 93.5 (3e-4) & \textbf{12.6} (3e-5)  &\textbf{.596} (2e-6) & \textbf{.412} (2e-6)
\\
\hline
prob & 86.3 (6e-4) & 22.4 (2e-4)  & 10.5 (6e-5)  & 5.13 (2e-5)
\\
\hline
dens & \textbf{51.5} (2e-4) & 21.7 (7e-5)   &  13.1 (3e-5)  & 13.2 (2e-5)
\\
\hline
%%%%%%%%%%%%%%%%%%%%%%%%
\multicolumn{5}{|c|}{rank-2 state, MSEs$\times 10^{3}$}
%%%%%%%%%%%%%%%%%%%%%%%%
\\
\hline
Inversion & 16.8 (8e-4)  & 15.9 (3e-4) & 15.9 (1e-4) & 15.8 (7e-5)
\\
\hline
Thresholding & 14.9 (3e-4)  & 15.5 (7e-5) & 15.5 (9e-6) & 15.5 (7e-6)
\\
\hline
prob &\textbf{9.29} (2e-3) & \textbf{7.90} (1e-3)  &\textbf{8.46} (1e-3) & \textbf{7.84} (8e-4)
\\
\hline
dens & 14.5 (3e-4) & 14.6 (3e-4)   & 14.4 (3e-4)   & 14.5 (4e-4)
\\
\hline
%%%%%%%%%%%%%%%%%%%%%%%%%%%%
\multicolumn{5}{|c|}{approximate rank-2 state, MSEs$\times 10^{3}$}
%%%%%%%%%%%%%%%%%%%%%%%%%%%
\\
\hline
Inversion & 15.9 (8e-4) & 15.4 (2e-4) & 15.3 (1e-4) & 15.2 (4e-5)
\\
\hline
Thresholding & 14.3 (2e-4)  &  14.2 (3e-4) & 15.0 (1e-5) & 15.0 (6e-6)
\\
\hline
prob & \textbf{8.88} (9e-4)  &\textbf{7.68} (2e-3) & \textbf{8.11} (1e-3) &\textbf{ 7.39} (1e-3)
\\
\hline
dens & 13.9 (4e-4) &  15.1 (2e-4) & 14.2 (3e-4)   & 14.2 (2e-4)
\\
\hline
%%%%%%%%%%%%%%%%%%%%%%
\multicolumn{5}{|c|}{maximal mixed state, MSEs$\times 10^{4}$}
%%%%%%%%%%%%%%%%%%%%%%%%
\\
\hline
Inversion & 15.9 (4e-4) & 6.57 (7e-5)  & 5.09 (5e-5) & 4.76 (2e-5)
\\
\hline
Thresholding & \textbf{4.67} (9e-5) & 5.59 (5e-5)  & 5.34 (8e-5) & 6.06 (8e-5)
\\
\hline
prob & 5.44 (2e-4) & \textbf{3.37} (8e-5)  & \textbf{3.31} (8e-5) & \textbf{3.20} (8e-5)
\\
\hline
dens & 5.72 (9e-5) & 4.47  (6e-5)   & 4.56 (4e-5)   & 4.24 (2e-5)
\\
\hline\hline
\end{tabular}
\end{table}
%%%%%%%%%%%%%%%
%%%%%%%%%%%%%%%
%%%%%%%%%%%%%%%
%%%%%%%%%%%%%%%

\begin{table}[H]
\caption{MSEs for $ n=3 $  (together with standard deviations)}
\label{tablen_3}
\begin{tabular}{| l | c | c | c | c| }
\hline\hline
                    & $m = 20$ & $m = 200$ & $m = 1000$ & $m = 2000$
\\
\hline
\multicolumn{5}{|c|}{pure state, MSEs$\times 10^{4}$}
\\
\hline
Inversion & 39.5 (9e-4) & 3.17 (9e-5)  & .559 (1e-5) & .343 (1e-5)
\\
\hline
Thresholding & 21.4 (6e-4) & \textbf{2.26} (1e-4) & \textbf{.196} (1e-5) & \textbf{.152} (1e-5)
\\
\hline
prob & 40.3 (2e-2)  & 5.79 (4e-4) & 2.95 (2e-4)  & 1.78 (1e-4)
\\
\hline
dens & \textbf{12.8} (5e-4) & 2.73 (2e-4) & 1.24 (4e-5)  &  1.07 (4e-5)
\\
\hline
%%%%%%%%%%%%%%%%%%
\multicolumn{5}{|c|}{rank-2 state, MSEs$\times 10^{2}$}
%%%%%%%%%%%%%%%%%%%%%
\\
\hline
Inversion & 3.69 (3e-3)  & 3.35 (6e-4) & 3.32 (4e-4) & 3.31 (2e-4)
\\
\hline
Thresholding & 2.94 (1e-3) & 3.05 (2e-4) & 3.04 (6e-5) & 3.05 (5e-5)
\\
\hline
prob &\textbf{1.91} (5e-3) & \textbf{1.17} (3e-3)  &\textbf{ 1.18} (3e-3) & \textbf{1.14} (2e-3)
\\
\hline
dens & 2.83 (8e-4) & 2.89 (3e-4)   & 2.89 (3e-4)   & 3.00 (1e-4)
\\
\hline
%%%%%%%%%%%%%%%%%%%%%%%%%
\multicolumn{5}{|c|}{approximate rank-2 state, MSEs$\times 10^{2}$}
%%%%%%%%%%%%%%%%%%%%%%%%%%%
\\
\hline
Inversion & 3.33 (2e-4) & 3.22 (8e-4) & 3.19 (3e-4) & 3.18 (2e-4)
\\
\hline
Thresholding & 2.81 (1e-3) & 2.96 (1e-4) & 2.97 (8e-5) & 2.97 (9e-5)
\\
\hline
prob & \textbf{1.10} (5e-3) & \textbf{.551} (5e-3)  &\textbf{.189} (2e-3) & \textbf{.113} (1e-3)
\\
\hline
dens & 2.74 (6e-4) &  2.88 (3e-4)  & 2.91 (3e-4)  & 2.91 (2e-4) 
\\
\hline
%%%%%%%%%%%%%%%%%%%%%%%%%%%
\multicolumn{5}{|c|}{maximal mixed state, MSEs$\times 10^{3}$}
%%%%%%%%%%%%%%%%%%%%%%%%%
\\
\hline
Inversion & 6.98 (2e-3) & 3.19 (4e-4) & 2.88 (2e-4) & 3.01 (1e-4)
\\
\hline
Thresholding & 4.41 (6e-4) & 3.26 (6e-4) & 3.19 (2e-4) & 3.29 (1e-4)
\\
\hline
prob & 3.63 (1e-3) &\textbf{2.70} (7e-4) & \textbf{2.28} (7e-4) &  \textbf{2.29} (1e-3)
\\
\hline
dens & \textbf{3.18} (6e-4) & 2.99 (4e-4) & 2.90 (2e-4)  & 3.04 (1e-4)
\\
\hline\hline
\end{tabular}
\end{table}

\begin{table}[H]
\caption{MSEs for $ n=2 $  (together with standard deviations)}
\label{tablen_2}
\begin{tabular}{| l | c | c | c | c |}
\hline\hline
    & $m = 20$ & $m = 200$ & $m = 1000$ & $m = 2000$
\\
\hline  \hline
  %%%%%%%%%%%%%%%%%%%%%%%
\multicolumn{5}{|c|}{pure state, MSEs$\times 10^{4}$}
%%%%%%%%%%%%%%%%%%%%%%%%
\\
\hline
Inversion & 61.9 (3e-3)  & 9.22 (5e-4)  & .802 (4e-5) & .772 (6e-5)
\\
\hline
Thresholding & \textbf{49.4} (3e-3) & \textbf{4.06} (3e-4) & \textbf{.737} (4e-5) & \textbf{.356} (2e-5)
\\
\hline
prob & 102 (8e-3) & 39.7 (2e-3)  & 9.37 (8e-4)  & 7.19 (5e-4)
\\
\hline
dens & 52.2 (3e-3) & 7.57  (5e-4)   &  1.91 (9e-5)  & 1.08 (2e-5)
\\
\hline 
  %%%%%%%%%%%%%%%%%%%%%%%
\multicolumn{5}{|c|}{rank-2 state, MSEs$\times 10^{2}$}
%%%%%%%%%%%%%%%%%%%%%%%%
\\
\hline
Inversion & 8.24 (2e-2) & 7.91 (3.2e-3)  & 7.81 (2e-3) & 7.74 (7e-4)
\\
\hline
Thresholding & 5.13 (3e-3) & 5.34 (1.1e-3)  & 5.32 (5e-4) & 5.33 (4e-4)
\\
\hline
prob & \textbf{2.62} (2e-2) & \textbf{1.77} (7.4e-3)  &\textbf{1.79} (8e-3) & \textbf{1.73} (5e-3)
\\
\hline
dens & 4.53 (3e-3) & 5.20 (1.5e-3)   & 5.24 (9e-4)   & 5.24 (9e-4)
\\
\hline  %%%%%%%%%%%%%%%%%%%%%%%
\multicolumn{5}{|c|}{approximate rank-2 state, MSEs$\times 10^{2}$}
%%%%%%%%%%%%%%%%%%%%%%%%
\\
\hline
Inversion & 8.12 (2e-2) &  7.54 (4e-3) & 7.54 (1.2e-3) & 7.56 (6e-4)
\\
\hline
Thresholding & 4.95 (4e-3) & 5.19 (8e-4)  & 5.23 (5e-4) & 5.22 (4e-4)
\\
\hline
prob & \textbf{2.69} (2e-2) &\textbf{1.82} (1.1e-2) &\textbf{1.52} (6e-3) &\textbf{1.58} (6e-3)
\\
\hline
dens & 4.40 (4e-3) & 5.02 (1.3e-3) & 5.11 (1e-3) & 5.15 (6e-4)
\\
\hline
%%%%%%%%%%%%%%%%%%%%%%%
\multicolumn{5}{|c|}{maximal state, MSEs$\times 10^{2}$}
%%%%%%%%%%%%%%%%%%%%%%%%
\\
\hline
Inversion & 3.03 (9e-3) & 2.12 (2e-3) & 2.11 (2e-3)  & 2.11 (1e-3)
\\
\hline
Thresholding & 2.78 (8e-3) & 2.36 (2e-3) & 2.21 (2e-3) & 2.25 (1e-3)
\\
\hline
prob & 2.32 (2e-2) & \textbf{1.15} (5e-3) & \textbf{1.19} (5e-3) & \textbf{1.07} (4e-3)
\\
\hline
dens & \textbf{2.30} (6e-3) & 2.11 (2e-3)  & 2.06 (2e-3)  & 2.09 (1e-3)
\\
\hline\hline
\end{tabular}
\end{table}
%%%%%%%%%%%%%%%%%%%%%%%%
%%%%%%%%%%%%%%%%%%%%%%%%
%%%%%%%%%%%%%%%%%%%%%%%

\subsection{Real data tests}
The experiments performed to produce the data is
explained in \cite{barreiro2010experimental}. The
data was kindly provided by M.~Gu{\c{t}}{\u{a}}
and T. Monz. It had been used in \cite{alquier2013rank,guctua2012rank}.
We apply two proposed estimators to the real data set
of a system of 4 ions which is Smolin state further
manipulated. In Figure \ref{realdata} we plot the
eigenvalues of the inversion estimator and our ones.
%%%%%%%%%%%%%%%
%%%%%%%%%%%%%%%
\begin{figure}[!ht]
\centering
\includegraphics[scale=.47]{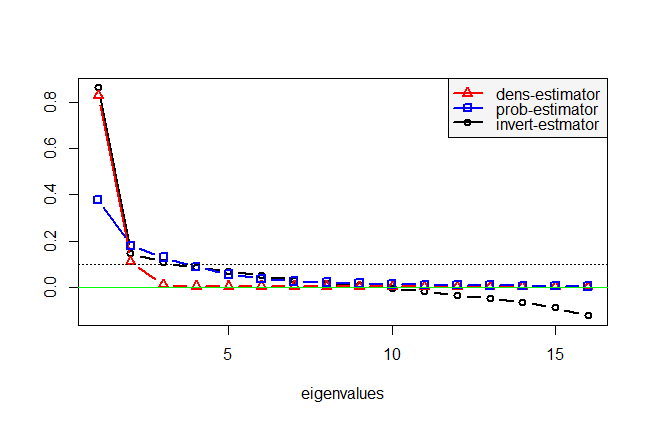}
\caption{eigenvalues plots for real data test with
$ n=4 $}
\label{realdata}
\end{figure}
%%%%%%%%%%%%%%%
%%%%%%%%%%%%%%%
Note that the distribution of the eigenvalues of the three estimators
are rather different. Still, it seems that all estimators return results
compatible with a rank-2 state.

\section{Discussion and conclusion}
\label{conclusion}
We propose a novel prior and introduce two pseudo-Bayesian estimators for
the density matrix: the dens-estimator and the prob-estimator.
The prob-estimator reaches the best up-to-date rate of
convergence in the low-rank case. On the other hand, computation of
the dens-estimator is an easier task. In practice, we
recommend the prob-estimator. However, in cases where
the MCMC shows activities of lacking of convergence,
the dens-estimator can be used as a reasonable alternative.

Note also that the prob-estimator can be extended to the
incomplete measurement case. We consider the
(incomplete) pseudo-likelihood as
\begin{align*}
 \ell^{prob-incomplete}(\nu,\mathcal{D}) = \sum_{\mathbf{a}\in \mathcal{A}}
\sum_{\mathbf{s}\in \mathcal{R}^n}
\left[{\rm Tr}(\nu P_\mathbf{s}^\mathbf{a})
-\hat{p}_{\mathbf{a},\mathbf{s}}\right]^2,
\end{align*}
where $ \mathcal{A} \varsubsetneq \mathcal{E}^n $. The study
in this case will be the object of future works.

Open questions include faster algorithms based on
optimization (in the spirit of \cite{alquier2015properties}).
Also, from a theoretical perspective, the most important
question is the minimax lower bound.

\section*{Acknowledgements}
Both authors gratefully acknowledge financial support from GENES and
by the French National Research Agency (ANR) under the grant
Labex Ecodec (ANR-11- LABEX-0047). P.Alquier gratefully acknowledges
financial support from the research programme {\it New Challenges for New
Data} from LCL and GENES, hosted by the {\it Fondation du Risque}.

\newpage
\appendix

\section{Proofs}
\label{Appendix}

We first remind here a version of Hoeffding's inequality
for bounded random variables.
\begin{lemma}
\label{bernstein ine}
Let $Y_i, i = 1,\ldots,n$ be $n$ independent random
variables with $|Y_i|\leq b$ a.s., and
 $\mathbb{E}(Y_i)=0$. Then, for any $\lambda>0$,
 $$ \mathbb{E}\exp\left(\frac{\lambda}{n}
 \sum_{i=1}^{n} Y_i\right)
        \leq
 \exp\left(\frac{\lambda^2 b^2}{8n} \right) .$$
\end{lemma}

\subsection{Preliminary lemmas for the proof of Theorem~\ref{thrm: rate 1}}

\begin{lemma}
\label{lem app bern 1}
 For any $\lambda>0$, we have
 $$\mathbb{E}\exp\left(\lambda \left<p_\nu - p^0,
 p^0-\hat{p}\right>_F\right)
  \leq \exp\left[ \frac{\lambda^2}{4m}\|p^0
  -p_\nu\|_F^2  \right],
  $$
$$\mathbb{E}\exp\left(-\lambda \left<p_\nu - p^0,
p^0-\hat{p}\right>_F\right)
  \leq
  \exp\left[ \frac{\lambda^2}{4m}\|p^0
  -p_\nu\|_F^2  \right]. $$
\end{lemma}
\begin{proof}
 First inequality:
 \begin{align*}
  \mathbb{E}\exp   & \left(\lambda
  \left<p_\nu - p^0,p^0-\hat{p}\right>_F\right)
\\
  & = \mathbb{E}\exp\left(\lambda \sum_{a\in
  \mathcal{E}^n}\sum_{s \in \mathcal{R}^n}
  \underbrace{[{\rm Tr}(\nu P_s^a)
   - p^0_{a,s}]}_{=: c(a,s)}
      [p^0_{a,s}-\hat{p}_{a,s}] \right)
  \\
  & = \prod_{a\in   \mathcal{E}^n} \mathbb{E}
  \exp\left(\lambda  \sum_{s \in \mathcal{R}^n} c(a,s)
      \left[p^0_{a,s}-\frac{1}{m}\sum_{i=1}^m
\mathbf{1}(R_i^a=s)\right] \right)
 \\
 & = \prod_{a\in \mathcal{E}^n} \mathbb{E}
 \exp\Bigg(\frac{\lambda}{m}
 \sum_{i=1}^m \underbrace{\left[ \sum_{s
 \in \mathcal{R}^n} c(a,s) \{p^0_{a,s} -
 \mathbf{1}(R_i^a=s)\} \right]}_{=: Y_{i,a}}\Bigg)
 \end{align*}
We have that $\mathbb{E}(Y_{i,a})=0$.
Then, using Cauchy-Schwartz inequality
 \begin{multline*}
Y_{i,a}^2
   \leq
\left(\sum_{s \in \mathcal{R}^n}
  c(a,s)^2\right)
   \left(\sum_{s \in \mathcal{R}^n} [p^0_{a,s}
   -\mathbf{1}(R_i^a=s)]^2\right)
   \\
   \leq
\left(\sum_{s \in \mathcal{R}^n} c(a,s)^2\right)
  \left(\sum_{s \in \mathcal{R}^n}  |p^0_{a,s} -
   \mathbf{1}(R_i^a=s) | \right)
 \leq 2 \left(\sum_{s \in \mathcal{R}^n} c(a,s)^2\right).
 \end{multline*}
 So we can apply Hoeffding's inequality (Lemma~\ref{bernstein ine}):
\begin{align*}
 \prod_{a\in \mathcal{E}^n}
  \mathbb{E}\exp\left(\frac{\lambda}{m}
 \sum_{i=1}^m Y_{i,a} \right)
  &
  \leq
   \prod_{a\in \mathcal{E}^n}
     \exp\left[   \frac{2\lambda^2}{8m}
\left(\sum_{s \in \mathcal{R}^n} c(a,s)^2\right)\right]
  \\
  & \leq  \exp\left[
   \frac{\lambda^2}{4m} \|p-p_\nu\|_F^2  \right].
\end{align*}
 Second inequality: same proof, just replace $Y_{i,a}$ by $-Y_{i,a}$.
\end{proof}

\begin{lemma}
\label{lem3}
For $ \lambda >0 $, we have
  \begin{align}
 \label{apbern1}
 \mathbb{E}\exp\left\{
 \lambda \left(\|p_\nu - \hat{p}\|_F^2
      - \|p^0-\hat{p}\|_F^2\right)
 - \lambda\left[1+  \frac{\lambda}{m} \right]\|p^0-p_\nu\|_F^2
      \right\}     \leq 1,
 \\
       \mathbb{E}\exp\left\{\lambda\left[1-
       \frac{\lambda}{m}   \right]\|p^0-p_\nu\|_F^2
    -\lambda \left(\|p_\nu - \hat{p}\|_F^2
    - \|p^0-\hat{p}\|_F^2\right)\right\} \leq 1 .
     \label{apbern2}
 \end{align}
\end{lemma}

\begin{proof}
 Proof of the first inequality:
 \begin{align*}
\mathbb{E}\exp & \left\{ \lambda \left(\|p_\nu-
\hat{p}\|_F^2 - \|p-\hat{p}\|_F^2\right) \right\}
 \\
&     = \mathbb{E}\exp\left\{ \lambda
\left<p_\nu - p^0,p_\nu+p^0 -2\hat{p}\right>_F \right\}
 \\
& = \mathbb{E}\exp\left\{ \lambda \|p_\nu-p^0\|_F^2
+ 2\lambda \left<p_\nu - p^0,p^0-\hat{p}\right>_F \right\}
 \\
& = \exp\left(\lambda \|p_\nu-p^0\|_F^2\right)
 \mathbb{E}\exp\left\{  2\lambda \left<p_{\nu}
 -p^0,p^0-\hat{p}\right>_F \right\}
   \\
&   \leq  \exp\left(\lambda \|p_\nu-p^0\|_F^2\right)
 \exp\left\{ \frac{\lambda^2}{m}
  \|p_{\nu}-p^0\|_F^2\right\}
 \end{align*}
thanks to Lemma~\ref{lem app bern 1}.
The proof of the second inequality is similar.
\end{proof}

Using Lemma~\ref{lem3}, we derive an
empirical PAC-Bayes bound for the estimator.
\begin{lemma}
\label{PAC-bound:empirical}
For $\lambda>0$ s.t. $ \frac{\lambda}{m} <1$,
with prob. $ 1- \epsilon/2, \epsilon\in(0,1) $,
for any distribution $\hat{\pi}$, we have:
\begin{align*}
 \int \|p_\nu-p^0\|_F^2 \tilde{\pi}_{\lambda} ({\rm d}\nu)
 \leq \frac{ \int\|p_\nu-\hat{p}\|_F^2
 \hat{\pi} ({\rm d}\nu)  - \|p^0-\hat{p}\|_F^2
 +\frac{\mathcal{K}(\tilde{\pi}_{\lambda},\pi)+
 \log\left(\frac{2}{\epsilon}
  \right)}{\lambda}}{1- \frac{\lambda}{m}}.
\end{align*}
\end{lemma}

\begin{proof}
We rewrite (\ref{apbern2}) in Lemma~\ref{lem3} as
follows
\begin{align*}
\int\mathbb{E}\exp \Bigg\{ \lambda
\left[1 - \frac{\lambda}{m}\right] \|p^0 -p_\nu\|_F^2
 - \lambda  \left(\| p_\nu - \hat{p} \|_F^2
 - \| p^0 - \hat{p}\|_F^2\right)  \Bigg\} \pi (d\nu)
  \leq 1.
 \end{align*}
By using Fubini's theorem
\begin{align*}
\mathbb{E}\int\exp \Bigg\{ \lambda
\left[1 - \frac{\lambda}{m} \right]  \|p^0-p_\nu\|_F^2
- \lambda  \left(\| p_\nu - \hat{p} \|_F^2
 - \| p^0- \hat{p}\|_F^2\right)  \Bigg\} \pi (d\nu)
  \leq 1.
 \end{align*}
Now, using~\cite[Lemma 1.1.3]{catonibook},
for any distribution $ \hat{\pi} $, we have
\begin{align*}
\mathbb{E}\exp\sup_{\hat{\pi}}\Bigg\{
 \lambda\left[ 1 - \frac{\lambda}{m} \right]
 \int \|p^0-p_\nu\|_F^2 \hat{\pi}(d\nu)
  - \log\left(2 / \epsilon \right)
   - \mathcal{K}(\hat{\pi},\pi) \hspace*{1cm}
\\
- \lambda  \left(\int \| p_{\nu} - \hat{p}\|_F^2   \hat{\pi}(d\nu)
 - \| p^0- \hat{\rho}\|_F^2\right)  \Bigg\}
  \leq \frac{\epsilon}{2}
 \end{align*}
and with $ \mathbf{1}_{\mathbf{R}_+}(x) \leq \exp(x) $,
one has
\begin{align*}
\mathbb{P}\Bigg\{ \sup_{\hat{\pi}} \Bigg[
\lambda\left[ 1 - \frac{\lambda}{m} \right]
 \int \|p^0 -p_\nu\|_F^2 \hat{\pi}(d\nu)
  - \log\left(2 / \epsilon \right)
   - \mathcal{K}(\hat{\pi},\pi) \hspace*{1cm}
\\
- \lambda  \left(\int \| p_{\nu} - \hat{p}\|_F^2   \hat{\pi}(d\nu)
 - \| p^0- \hat{\rho}\|_F^2\right)  \Bigg] \geq 0  \Bigg\}
  \leq \frac{\epsilon}{2}.
 \end{align*}
Taking the complementary yields successfully the results.
\end{proof}

The following lemma give a theoretical PAC-Bayes
bound for the estimator.
\begin{lemma}
 \label{bound1}
For $ \lambda>0 $ s.t $ \frac{\lambda}{m}<1$, with
probability  $1-\epsilon$ we have:
 \begin{align}
 \label{bound1-equation}
 \int \|p_\nu-p^0\|_F^2\hat{\pi}^{prob}_{\lambda}
 ({\rm d}\nu)   \leq \inf_{\hat{\pi}} \frac{
  \left[1+ \frac{\lambda}{m} \right]\int \|p_\nu
  -p^0\|_F^2\tilde{\pi}({\rm d}\nu)
           +\frac{2\mathcal{K}(\hat{\pi},\pi)
           +   2\log\left(\frac{2}{\epsilon}
            \right)}{\lambda}  }{1 - \frac{\lambda}{m} }
\end{align}
and
\begin{align}\label{oracle main}
\int \|\nu-\rho^0\|_F^2\hat{\pi}^{prob}_{\lambda}
   ({\rm d}\nu)   \leq        \inf_{\hat{\pi}} \frac{
  3^n \left[1+ \frac{\lambda}{m} \right]
  \int \|\nu-\rho^0\|_F^2\hat{\pi}({\rm d}\nu)
           +\frac{2\mathcal{K}(\hat{\pi},\pi)
           +  2\log\left(\frac{2}{\epsilon}
            \right)}{2^n \lambda}    }{1- \frac{\lambda}{m}}.
 \end{align}
\end{lemma}

\begin{proof}
\label{proofbound1}
Using the same proof of
Lemma~\ref{PAC-bound:empirical} for inequality
(\ref{apbern1}) in Lemma~\ref{lem3}, we obtain
with probability at least
$ 1- \epsilon/2, \epsilon\in(0,1) $,
for any distribution $\hat{\pi}$ that
\begin{align*}
\int \| p^0- \hat{p}\|_F^2   \hat{\pi}(d\nu)
\leq
\left[ 1 + \frac{\lambda}{m}\right] \int \| p_{\nu}- p^0\|_F^2
 \hat{\pi}(d\nu)  + \| p^0- \hat{p}\|_F^2
+ \frac{\mathcal{K}(\hat{\pi},\pi) +
\log(\frac{2}{\epsilon})}{\lambda}
 \end{align*}
With a union argument, combining the Lemma~\ref{PAC-bound:empirical}
and the above inequality yields the following inequality
with probability at least $ 1-\epsilon, \epsilon
 \in (0,1) $, for any $ \hat{\pi} $
 \begin{align*}
 \int \|p_{\nu} - p^0\|_F^2\hat{\pi}({\rm d}\nu)
\leq
 \frac{ \left[1+\frac{\lambda}{m} \right]
\int \|p_{\nu}-p^0\|_F^2\hat{\pi}({\rm d}\nu)
  +\frac{2\mathcal{K}(\hat{\pi},\pi)  +
   2\log(2/\epsilon )}{\lambda} }{1-\frac{\lambda}{m}}
 \end{align*}
Taking $ \tilde{\pi}^{prob}_{\lambda} $
(once again,~\cite[Lemma 1.1.3]{catonibook})
be the minimizer of the right hand side of the
above inequality, we obtain \eqref{bound1-equation}.

Moreover,  in~\cite[equation~(5)]{alquier2013rank}
states that, for any $\nu$:
$$ p_{\nu} = \mathbf{P} \nu  $$
for some operator $\mathbf{P}$. Therefore
$$ \|p_{\nu}-p^0\|_F^2 =
 \| \mathbf{P} (\nu-\rho^0) \|_F^2 .$$
The eigenvalues of $\mathbf{P}^T\mathbf{P}$ are known,
they range between $2^n$ and $3^n2^n$ according
to~\cite[Proposition 1]{alquier2013rank}.
Thus, for any $\nu$,
 $$ 2^n \|\nu-\rho^0\|_F^2 \leq \|p_\nu-p^0\|_F^2
  \leq 6^n \|\nu-\rho^0\|_F^2 $$
 and so we obtain (\ref{oracle main}).
\end{proof}

In the following, we will consider $ \hat{\pi} $
as a restriction of the prior to a local set around the
true density matrix $ \rho^0 $. This allows us to obtain
an explicit bound of the left hand side
of (\ref{oracle main}). Let $ \rho^0 =
U\Lambda U^{\dagger} $ be the spectral decomposition
of $\rho^0$.
\begin{definition}
\label{definition_pihat}
Let $r = \#\{ i : \Lambda_i > \delta \},$
with small $\delta\in [0,1)$. Take
$$
\tilde{\pi}_{c}(du,dv) \propto \mathbf{1} (
\forall i: |v_i - \Lambda_i| \leq \delta; \forall i = 1,\dots,r: \| u_i -
U_i \|_F \leq c) \pi (du,dv)
$$
\end{definition}
Note that we have $ r \leq {\rm rank}(\rho^0) $.
\begin{lemma}
\label{lemma_kl}
We have
\begin{align}
\label{estimate norm 1}
  \int  \| u^{\dagger}vu - \rho^0 \|^2_F
 \tilde{\pi}_{c}(du,dv)
  \leq
  (3d\delta + 2 rc)^2 .
\end{align}
And under the \hyperref[A1]{Assumption 1}
\begin{align}
  \mathcal{K}( \tilde{\pi}_{c},\pi)
 \leq  a rd\log(\frac{1}{c}) +
C_{D_1,D_2}d(\log(d)+ \log(\frac{1}{\delta}))
 \label{KL1}
\end{align}
where $ a $ is a universal constant and where
$C_{D_1,D_2}$ depends only on $D_1$ and $D_2$.
\end{lemma}

\begin{proof}
Firstly
\begin{align*}
\| uvu^{\dagger} - \rho^0 \|^2_F
\leq\biggl( \| uvu^{\dagger} - u\Lambda u^{\dagger} \|_F
+
\| u\Lambda u^{\dagger} -  U\Lambda U^{\dagger}\|_F
\biggl)^2
\end{align*}
and
\begin{align*}
  \| uvu^{\dagger} - u\Lambda u^{\dagger} \|_F
 & \leq \sum_{i} |v_i -\Lambda_i| \|u_i
 u_i^{\dagger} \|_F  \leq d \delta,
\\
 \| u\Lambda u^{\dagger} - U\Lambda U^{\dagger} \|_F
& \leq\sum_{i}\Lambda_i\|u_i u_i^{\dagger} - U_i
U_i^{\dagger} \|_F
\\
& \leq \sum_{i:\Lambda_i>\delta}
 (\|u_i u_i^{\dagger} - u_i U_i^{\dagger} \|_F
 +  \|u_i U_i^{\dagger}- U_iU_i^{\dagger} \|_F)
 \\
 & \hspace*{2cm} + \delta\sum_{i:\Lambda_i\leq\delta}
 (\|u_i u_i^{\dagger}\|_F
  +\| U_i U_i^{\dagger} \|_F)
\\
& \leq 2rc+2\delta(d-r)\leq 2rc+2\delta d   ,
\end{align*}
so we obtain (\ref{estimate norm 1}).

Now, the Kullback-Leibler term
\begin{align*}
\mathcal{K}( \tilde{\pi}_{c},\pi)=
& \log\frac{1}{\pi(\{ u,v: \forall i: |v_i - \Lambda_i| \leq
 c; \forall i = 1,r: \| u_i - U_i \|_F \leq \delta \})}
\\
= & \log\frac{1}{\pi(\{\forall i : |v_i - \Lambda_i |
  \leq \delta\})}
  + \log\frac{1}{ \pi \left( \left\lbrace
   \forall i = 1,r : \|u_{i.} - U_{i.}\|_{F}
   \leq c \right\rbrace \right)}.
\end{align*}
The first log term
\begin{align*}
 \pi \left( \left\lbrace \forall i = 1,r :
 \|u_{i.} - U_{i.}\|_{F} \leq  c \right\rbrace \right)
 &
 \geq \prod_{i =1}^r  \Bigg[  \dfrac{\pi^{(d-1)/2}
 (c/2)^{d-1} }{ \Gamma(\frac{d-1}{2}+1)}     \Bigg/
   \dfrac{2 \pi^{(d+1)/2}}{\Gamma(\frac{d+1}{2})}  \Bigg]  , d = 2^n
 \\
 & \geq  \Bigg[ \dfrac{c^{d-1}}{2^d\pi}  \Bigg]^r
  \geq   \dfrac{c^{r(d-1)}}{2^{4rd}}   .
\end{align*}
Note for the above calculation: it is greater or equal to the
volume of the (d-1)-"circle" with radius $ c/2 $
over the surface area of the $d$-``unit-sphere".

The second log term in the Kullback-Leibler term
\begin{align*}
\pi(\{\forall i : |v_i - \Lambda_i | \leq \delta\})
 & =
\frac{\Gamma(D_1)}{\prod_{i=1}^d
\Gamma(\alpha_i)} \prod_{i=1}^d
\int_{\max(\Lambda_i - \delta,0)}^{\min(\Lambda_i + \delta,1)} v_i^{\alpha_i -1}  dv_i
\\
 & \geq \Gamma(D_1) \delta^d \prod_{i=1}^d\alpha_i
\geq
C_{D_1} \delta^d \  e^{-D_2d \log(d)}
  \end{align*}
 for some constant $C_{D_1}$ that depends only on $D_1$.
Since $ \alpha_i \leq 1 $ for every $ i$, we
can lower bound the integrand by $1$ and also
$ \alpha_i\Gamma(\alpha_i) = \Gamma(\alpha_i+1)
\leq 1 $. The interval of integration contains at least
an interval of length $ \delta$. This trick was
presented in \cite[Lemma 6.1, page 518]{ghosal2000}

Thus, we obtain
\begin{align*}
\mathcal{K}( \tilde{\pi}_{c},\pi)
&  \leq     \log \dfrac{2^{4rd}}{c^{r(d-1)}}
+  \log\left( \frac{e^{D_2d \log(d)}}{C_{D_1} \delta^d}  \right)
\\
& \leq  a rd\log(\frac{1}{c}) +
C_{D_1,D_2}d(\log(d)+ \log(\frac{1}{\delta}))
\end{align*}
for some absolute constant $a$ and where $C'_{D_1,D_2}$
depends only on $D_1$ and $D_2$.
\end{proof}

\subsection{Proof of Theorem~\ref{thrm: rate 1}}

\begin{proof}[Proof of Theorem~\ref{thrm: rate 1}]
Substituting~(\ref{KL1}),(\ref{estimate norm 1})
into~(\ref{oracle main}), we obtain
\begin{align*}
  \int \|\nu-\rho^0\|_F^2\tilde{\pi}_{\lambda}({\rm d}\nu)
\leq
 \inf_{c} &\Bigg\{ \frac{
 3^n \left[1+\frac{\lambda}{m}\right]
  (3d\delta + 2 rc)^2}{1-\frac{\lambda}{m}}
\\
&  +  \frac{a rd\log(\frac{1}{c}) +
C_{D_1,D_2}d(\log(d)+ \log(\frac{1}{\delta}))
  +2\log(2/\epsilon
 )}{ \lambda 2^n [1-\frac{\lambda}{m}]} \Bigg\}.
\end{align*}
By taking $\delta =\frac{1}{d\sqrt{N}}, c = \sqrt{\frac{d}{rm 9^n}}, \lambda = m/2 $
leads to
\begin{align*}
  \int \|\nu-\rho^0\|_F^2\tilde{\pi}_{\lambda}({\rm d}\nu)
\leq
A\left( \frac{1}{m} + \frac{rd}{m3^n} \right)
 +  C'_{D_1,D_2}\frac{r\log(rm3^n/d)
+\log(m3^n) +\log(2/\epsilon )/2^n}{ m}
\end{align*}
for some absolute constant $A$.
Finally, by Jensen inequality, one has
\begin{align*}
\|\hat{\rho}_{\lambda} - \rho^0\|_F^2 \leq
  \int \|\nu-\rho^0\|_F^2\hat{\pi}_{\lambda}({\rm d}\nu).
\end{align*}
This completes the proof of the theorem.
\end{proof}

\subsection{Preliminary results for the proof of the Theorem~\ref{rate 2}}

Rewriting equation~(\ref{prob.fomala}), by
plugging~(\ref{Pauli expansion}) in, as follow
\begin{align*}
p_{\mathbf{a},\mathbf{s}} = \sum_{b\in \{ I,x,y,z \}^n }
\rho_b {\rm Trace} \left(\sigma_b \cdot
 P_{\mathbf{s}}^{\mathbf{a}} \right)
 =  \sum_{b\in \{ I,x,y,z \}^n }\rho_b \mathbf{P}_{(s,a),b}.
\end{align*}
Where $ \mathbf{P}_{(s,a),b}  = \prod_{j\neq E_b} s_j \mathbf{1}
(a_j = b_j) $ and $ E_b = \{ j\in \{1,\ldots,n\}: b_j = I \} $
, see~\cite{alquier2013rank} for technical details.
We are now ready to handle with the proofs.

\begin{lemma}
\label{lem apply Bern}
 For any $\lambda>0$, we have
 $$\mathbb{E}\exp\left(\lambda
  \left<\rho^0-\nu,\rho^0-\hat{\rho}\right>_F\right)
  \leq
  \exp\left[\frac{4 \lambda^2 }{m}
   \left(\frac{5}{3}\right)^n \|\nu-\rho^0\|_F^2
\right]  $$
 $$\mathbb{E}\exp\left(-\lambda
 \left<\rho^0-\nu,\rho^0-\hat{\rho}\right>_F\right)
  \leq
   \exp\left[\frac{4 \lambda^2 }{m}
   \left(\frac{5}{3}\right)^n \|\nu-\rho^0\|_F^2
\right] . $$
\end{lemma}

\begin{proof}
 First inequality
 \begin{align*}
 &  \mathbb{E}  \exp \left(\lambda \left<\rho^0-\nu,
		\rho^0 -\hat{\rho} \right>_F\right)
\\		
 & = \mathbb{E}\exp\left[\lambda
\sum_{b}    (\rho^0_b-\nu_b)
  (\rho^0_b-\hat{\rho}_b)Trace(\sigma_b\sigma^{\dagger}_b) \right]
  \\
& = \mathbb{E}\exp\left[d \lambda \sum_{b} (\rho^0_b-\nu_b)
 \sum_s \sum_a   \frac{\mathbf{P}_{(s,a),  b}}{3^{d(b)}2^n}
 (p^0_{a,s} -  \hat{p}_{a,s}) \right]
 \\
& = \prod_a \mathbb{E}\exp\left[\lambda \sum_{b} (\rho^0_b-\nu_b)
 \sum_s \frac{1}{m}  \sum_{i=1}^m
   \frac{\mathbf{P}_{(s,a),   b}}{3^{d(b)}}
   (p^0_{a,s} - \mathbf{1}_{R_i^a=s}) \right]
 \\
& = \prod_a \prod_i  \mathbb{E}\exp\Bigg[\frac{\lambda}{m}
   \underbrace{\sum_{b} (\rho^0_b-\nu_b) \sum_s
\frac{\mathbf{P}_{(s,a),   b}}{3^{d(b)}}(
p^0_{a,s}- \mathbf{1}_{R_i^a=s})}_{ : =Y_{i,a}}\Bigg].
 \end{align*}
 Remark that $\mathbb{E}(Y_{i,a})=0$. Also, from
the definitions above, the absolute value
$|\mathbf{P}_{(s,a),   b}|$ does not depend on $s$ so
 \begin{align*}
  |Y_{i,a}|     & \leq   \sum_{b} |\rho^0_b-\nu_b|
  \left|\frac{\mathbf{P}_{(s,a),   b}}{3^{d(b)}}\right|
  \sum_s         | p^0_{a,s} - \mathbf{1}_{R_i^a=s} |
  \\
  & \leq 2 \sum_{b} |\rho^0_b-\nu_b|
        \left|\frac{\mathbf{P}_{(s,a),   b}}{3^{d(b)}}\right|
 \leq \frac{2}{2^{n/2}} \sqrt{\sum_b (\rho^0_b - \nu_b)^2d  \sum_b
 \left( \frac{\mathbf{P}_{(s,a), b}}{3^{d(b)}} \right)^2 }
  \\
& \leq \frac{2 \|\nu-\rho^0\|_F}{ 2^{n/2}} \left(
\sum_b \frac{1}{3^{2d(b)}}
   \prod_{j\notin E_b}  \mathbf{1}_{a_j=b_j}  \right)^{1/2}
    \\
  & \leq \frac{2 \|\nu- \rho^0\|_F }{ 2^{n/2}}
 \left(\sum_{\ell=0}^n {n \choose \ell} \frac{1}{3^{2\ell}} \right)^{1/2}
    \\
  & \leq\frac{2 \|\nu- \rho^0\|_F}{2^{n/2}} \left( 1 + \frac{1}{9} \right)^{n/2}
  = 2 \|\nu- \rho^0\|_F\left( \frac{5}{9} \right)^{n/2}.
 \end{align*}
 So we can apply Hoeffding's inquality (Lemma~\ref{bernstein ine}):
\begin{align*}
 \prod_a \mathbb{E}\exp\left(\frac{\lambda}{m}
 \sum_{i=1}^m Y_{i,a} \right)
  & \leq    \exp\left[   \frac{\lambda^2 }{2m}
   \left(\frac{5}{3}\right)^n \|\nu-\rho^0\|_F^2
  \right].
\end{align*}
 Second inequality: same proof, just replace $Y_i(a)$ by $-Y_i(a)$.
\end{proof}

\begin{lemma}
\label{lem exp_bound 2}
 We have
 \begin{align*}
\mathbb{E}\exp \Big\{ \lambda\left[1 -
\frac{2\lambda}{m} \left(\frac{5}{3}\right)^n  \right]   \| \nu - \rho^0\|_F^2
-\lambda \left(\|\nu - \hat{\rho}\|_F^2 - \|\rho^0-\hat{\rho}
 \|_F^2\right)  \Big\}   \leq 1,
\\
  \mathbb{E}\exp \Big\{ \lambda
 \left(\| \nu - \hat{\rho}\|_F^2  - \| \rho^0- \hat{\rho}\|_F^2\right)
  - \lambda\left[  1 + \frac{2\lambda}{m}
  \left(\frac{5}{3}\right)^n \right]
  \|\nu-\rho^0\|_F^2    \Big\}  \leq 1 .
 \end{align*}
\end{lemma}

\begin{proof}
For the second inequality:
 \begin{align*}
\mathbb{E}\exp  &  \left\{ \lambda
 \left(\| \nu- \hat{\rho}\|_F^2  -  \| \rho^0 -\hat{\rho}\|_F^2\right) \right\}
\\
&     = \mathbb{E}\exp\left\{  \lambda
      \left< \nu -  \rho^0, \nu+ \rho^0- 2\hat{\rho}\right>_F \right\}
 \\
& = \mathbb{E}\exp\left\{ \lambda \| \nu- \rho^0\|_F^2
+ 2\lambda \left< \nu - \rho^0, \rho^0- \hat{\rho}\right>_F \right\}
 \\
& = \exp\left(\lambda \| \nu- \rho^0 \|_F^2\right) \mathbb{E}\exp\left\{
 2\lambda \left< \nu- \rho^0, \rho^0-\hat{\rho}\right>_F \right\}
   \\
&   \leq
   \exp\left(\lambda \| \nu- \rho^0 \|_F^2\right) \exp
\left\{\frac{2\lambda^2 }{m} \left(\frac{5}{3}\right)^n
     \|\nu-\rho^0\|_F^2  \right\}
 \end{align*}
 thanks to the Lemma~\ref{lem apply Bern}.
 The proof of the first inequality is similar.
\end{proof}

\begin{lemma}
\label{PACbound_2}
For $ \lambda>0 $ s.t $ \frac{2\lambda}{m} \left(\frac{5}{3}\right)^n <1$, with
probability at least $ 1-\epsilon, \epsilon \in (0,1) $,
 we have
 \begin{align}
  \int \|\nu-\rho^0\|_F^2\tilde{\pi}^{dens}_{\lambda}({\rm d}\nu)
\leq
 \inf_{\hat{\pi}} \frac{ \left[1+
 \frac{2\lambda }{m} \left(\frac{5}{3}\right)^n \right]
\int \|\nu-\rho^0\|_F^2\hat{\pi}({\rm d}\nu)
  +\frac{2\mathcal{K}(\hat{\pi},\pi)  +
   2\log(2/\epsilon )}{\lambda} }{1-\frac{2\lambda }{m} \left(\frac{5}{3}\right)^n}.
    \label{PAC bound 2}
 \end{align}
\end{lemma}

\begin{proof}
By using the results from the Lemma~\ref{lem exp_bound 2}, the proof is similar to the proof of
Lemma~\ref{bound1}
page~\pageref{proofbound1}.
\end{proof}

\subsection{Proof of Theorem~\ref{rate 2}}

\begin{proof}[Proof of Theorem~\ref{rate 2}]
Substituting~(\ref{KL1}),(\ref{estimate norm 1})
into~(\ref{PAC bound 2})
\begin{align*}
  \int \|\nu-\rho^0\|_F^2\hat{\pi}_{\lambda}({\rm d}\nu)
\leq
 \inf_{c} &\Bigg\{ \frac{\left[1+\frac{2\lambda}{m} \left(\frac{5}{3}\right)^n \right]
  (3d\delta + 2 rc)^2}{1-\frac{2\lambda}{m} \left(\frac{5}{3}\right)^n}
\\
 & +  \frac{a rd\log(\frac{1}{c}) +
C_{D_1,D_2}d(\log(d)+ \log(\frac{1}{\delta})) +2\log(2/\epsilon
 )}{\lambda [1-\frac{2\lambda}{m} \left(\frac{5}{3}\right)^n]} \Bigg\}.
\end{align*}
Taking $\delta = \frac{d}{N}, c = \sqrt{\frac{d}{rN}},
 \lambda = \frac{N}{5^{n}4} $
lead to
\begin{align*}
  \int \|\nu-\rho^0\|_F^2\hat{\pi}_{\lambda}({\rm d}\nu)
\leq
A'\frac{d^2 r}{N}  + C_{D_1,D_2} 5^{n}
\frac{rd\log(\frac{Nr}{d}) + d\log(\frac{N}{d})
   +2\log(2/\epsilon )}{N}
 \end{align*}
 for some constant $A'>0$.
Simultaneously, by Jensen inequality, one has
\begin{align*}
\|\hat{\rho}_{\lambda} - \rho^0\|_F^2 \leq
  \int \|\nu-\rho^0\|_F^2\hat{\pi}_{\lambda}({\rm d}\nu).
\end{align*}
This complete the proof of the theorem.
\end{proof}

\end{document}